\documentclass[11pt,english]{article}
\def\margin_comment#1{\marginpar{\sffamily{\tiny #1\par}\normalfont}}

\usepackage{graphicx}
\usepackage{setspace}
\date{}
\usepackage{amsmath,amssymb,amsthm}
\usepackage{babel}
 \newtheorem{thm}{Theorem}[section]
 \numberwithin{equation}{section} 
 \numberwithin{figure}{section} 
 \theoremstyle{plain}
  \newtheorem*{thm*}{Theorem}
 \theoremstyle{definition}
\theoremstyle{plain}
\newtheorem{thm_A}{Theorem}
 \newtheorem*{defn*}{Definition}
 \theoremstyle{plain}
 \newtheorem{prop}[thm]{Proposition} 
 \theoremstyle{remark}
 \newtheorem{ex}[thm]{Example}
 \theoremstyle{remark}
 
\theoremstyle{plain}
 
\theoremstyle{plain}
 \newtheorem{cor}[thm]{Corollary}
 \theoremstyle{plain}
 \newtheorem{lem}[thm]{Lemma} 
 \theoremstyle{definition}
 \newtheorem{defn}[thm]{Definition}
\newtheorem*{acknowledgment}{Acknowledgment}

\newcommand{\Div}[1]{\operatorname{Div}(#1)}

\AtBeginDocument{
  
}
\begin{document}
\title{Folding of set-theoretical solutions of the Yang-Baxter equation}

\author{Fabienne Chouraqui$^{*}$ and Eddy Godelle\footnote{Both authors are partially supported by the {\it Agence Nationale de la Recherche} ({\it projet
Th\'eorie de Garside}, ANR-08-BLAN-0269-03). The first author is also supported by the Affdu-Elsevier fellowship.}}
\maketitle

\begin{abstract}
 We establish a correspondence between the invariant subsets of a non-degenerate symmetric set-theoretical solution of the quantum Yang-Baxter equation and the parabolic subgroups of its structure group, equipped with its canonical Garside structure. Moreover, we introduce the notion of a foldable solution, which extends the one of a decomposable solution.
\end{abstract}

\section{Introduction}

The \emph{Quantum Yang-Baxter Equation} (\emph{QYBE} for short) is an important equation in the field of mathematical physics, and it lies in the foundation of the theory of quantum groups. Finding all the solutions of this equation is an important issue.

Let $V$ be a vector space with base~$X$ and $S: X\times X\to X\times X$ be a bijection.  The pair~$(X,S)$ is said to be a \emph{set-theoretical solution} of the QYBE if the linear operator~$R:V  \otimes V \rightarrow V  \otimes V$ induced by $S$ is a solution to the QYBE. The question of the classification of the set-theoretical solutions was raised by Drinfeld in~\cite{drinf}, and since has been the object of numerous recent articles~\cite{cedoetall,chou_art,etingof,gateva04,gateva08,gateva98}. In~\cite{etingof}, Etingof, Soloviev and Schedler, focus on  set-theoretical solutions that are non-degenerate and symmetric. In order to approach the classification problem, they introduce and use the notion of an \emph{invariant subset}. They associate a group called the \emph{structure group} to each non-degenerate and symmetric set-theoretical solution. In \cite{chou_art}, the first author establishes a one-to-one correspondence between non-degenerate and symmetric set-theoretical solutions of the QYBE and \emph{Garside group presentations} which satisfy some additional conditions. The notion of a Garside group is also the object of numerous articles~\cite{bessis,deh_francais,godelle,picantin} and is a natural generalisation of the notion of an Artin-Tits group. In particular the second author proves in~\cite{godelle} that the classical notion of a \emph{standard parabolic subgroup} of an Artin-Tits group can be extended to the general framework of Garside groups. The first objective of the present paper is to show that the two notions of an \emph{invariant subset} and of a \emph{standard parabolic subgroup} are deeply related in the context of non-degenerate symmetric set-theoretical solutions of the QYBE. More precisely (see next section for definitions and notations) we prove:

\begin{thm_A}\label{theo:theoprincip}
 Let $X$ be a finite set, and $(X,S)$ be a non-degenerate symmetric set-theoretical solution of the quantum Yang-Baxter equation. Let $G$ be the structure group of $(X,S)$. For $Y\subseteq X$, denote by $G_Y$ the subgroup of $G$ generated by $Y$. The map $Y\mapsto G_Y$ induces a one-to-one correspondance between the set of invariant non-degenerate subsets of $(X,S)$ and the set of standard parabolic subgroups of $G$.
\end{thm_A}

Indeed, in order to classify non-degenerate symmetric set-theoretical solution of the quantum Yang-Baxter equation, Etingof, Soloviev and Schedler introduced in~\cite{etingof} the notion of a \emph{decomposable solution}. This is a solution which is the union of two disjoint non-degenerate invariant subsets. In the last part of the article, we extend the notion of a \emph{folding of Coxeter graph} introduced by Crisp to study morphisms between Artin-Tits groups to the context of set-theoretical solutions of the QYBE. The notion of a \emph{foldable solution} can be seen as a generalisation of the notion of decomposable solution. We prove that every Garside subgroup, that verifies some obvious necessary condition, is associated to a  set-theoretical solution, and

\begin{thm_A}\label{theo:theoprincip2}
Let $X$ be a finite set, and $(X,S)$ be  a non-degenerate, symmetric set-theoretical solution of the QYBE. The pair~$(X,S)$ is decomposable if and only if it has a strong folding~$(X',S')$ which a trivial solution and such that $\# X' = 2$.
 \end{thm_A}

The paper is organized as follows.
In Section~$2$, we introduce the QYBE, the notion of a non-degenerate symmetric set-theoretical solution, and its structure group.
In Section $3$, we provide the necessary background on Garside groups and their parabolic subgroups. In Section $4$, we prove Theorem~\ref{theo:theoprincip}. In Section $5$, we introduce the notion of a folding and prove Theorem~\ref{theo:theoprincip2}.
\begin{acknowledgment}
The first author is  very grateful to Arye Juhasz for fruitful conversations. The authors are also grateful to Pavel Etingof and Travis Schedler for personnal communications~\cite{etingof2}.
\end{acknowledgment}

\section{Set-theoretical solutions  of the QYBE}
In this section, we introduce basic definitions and useful notions related to the Quantum Yang-Baxter Equation.  We follow~\cite{etingof} and refer to it for more details.\\

For all the section, we fix a vector space~$V$. The Quantum Yang-Baxter Equation on~$V$ is the equality $$R^{12}R^{13}R^{23}=R^{23}R^{13}R^{12}$$ of linear transformations on $V  \otimes V \otimes V$ where the indeterminate is a linear transformation~$R: V  \otimes V\to V  \otimes V$, and $R^{ij}$ means $R$ acting on the $i$th and $j$th components. A set-theoretical solution of this equation is a pair~$(X,S)$ such that $X$ is a basis for $V$, and $S : X \times X \rightarrow X \times X$ is a bijective map that induces a solution~$R$ of the QYBE.
\subsection{The structure group of a set-theoretical solution}
Consider a set-theoretical solution~$(X,S)$ of the QYBE. Following \cite{etingof}, for $x,y$ in $X$, we set  $S(x,y)=(g_{x}(y),f_{y}(x))$.
\begin{defn}
(i) The pair~$(X,S)$ is \emph{nondegenerate} if the maps $f_{x}$ and $g_{x}$ are bijections for any  $x\in X$.\\
(ii) The pair~$(X,S)$ is  \emph{involutive} if $S\circ S = Id_X$.\\
(iii) The pair~$(X,S)$ is \emph{braided} if $S$ satisfies the braid relation $S^{12}S^{23}S^{12}=S^{23}S^{12}S^{23}$, where the map $S^{ii+1}: X^{n} \rightarrow X^{n}$ is defined by $S^{ii+1}=id_{X^{i-1}} \times S\times id_{X^{n-i-1}} $, $i<n$.\\
(iv) The pair $(X,S)$ is \emph{symmetric} if  $(X,S)$ is involutive and braided.\\
\end{defn}

Let $\alpha:X \times X \rightarrow X\times X$ be the permutation map, i.e $\alpha(x,y)=(y,x)$, and let $R=\alpha \circ S$. The map $R$ is called the \emph{$R-$matrix corresponding to} $S$. Etingof, Soloviev and Schedler show in \cite{etingof}, that $(X,S)$ is a braided pair if and only
 if $R$ satisfies the QYBE, and that $(X,S)$ is a symmetric pair if and only if in addition $R$ satisfies the unitary condition $R^{21}R=1$. A solution ~$(X,S)$ is \emph{a trivial solution} if the maps $f_{x}$ and $g_{x}$ are the identity on $X$ for all  $x\in X$, that is  $S$ is the permutation map on $X\times X$. In the sequel we are interested in non-degenerate symmetric pairs.

\begin{defn} Assume~$(X,S)$ is non-degenerate and symmetric. The \emph{structure group} of $(X,S)$ is defined to be the group~$G(X,S)$ with the following group presentation:
$$\langle X\mid \forall x,y\in X,\ xy = g_x(y)f_y(x) \rangle$$
\end{defn}

Since the maps $g_x$ are bijective and $S$ is involutive, one can deduce that for each $x$ in $X$ there is a unique $y$ such that~$S(x,y) = (x,y)$. Therefore, the presentation of~$G(X,S)$ contains $\frac{n(n-1)}{2}$ non-trivial relations.

\begin{ex} \label{exemple:exesolu_et_gars} Let $X$ be the set~$\{x_1,\cdots, x_5\}$, and~$S$ be the map defined by $S(x_i,x_j)=(x_{\sigma_{i}(j)},x_{\tau_{j}(i)})$ where~$\sigma_i$ and $\tau_j$ are the following permutations on~$\{1,\cdots,5\}$: $\tau_{1}=\sigma_{1}= \tau_{3}=\sigma_{3}=(1,2,3,4)$; $\tau_{2}=\sigma_{2}= \tau_{4}=\sigma_{4}= (1,4,3,2)$ and
$\tau_5 = \sigma_5 =  id_{\{1,\cdots,5\}}$. A case-by-case analysis shows that $(X,S)$ is a non-degenerate symmetric solution. Its structure group is generated by the set~$X$ and defined by the~$10$ following relations: $$\begin{array}{ccccc}
  x^{2}_{1}=x^{2}_{2};&  x_{2}x_{5}=x_{5}x_{2};&x_{1}x_{2}=x_{3}x_{4};&  x_{1}x_{5}=x_{5}x_{1};&
x_{1}x_{3}=x_{4}x_{2};\\ x^{2}_{3}=x^{2}_{4};&
x_{2}x_{4}=x_{3}x_{1};& x_{3}x_{5}=x_{5}x_{3};&
x_{2}x_{1}=x_{4}x_{3};&x_{4}x_{5}=x_{5}x_{4}.
\end{array}$$
\end{ex}

\subsection{Decomposability of a solution}

Here we introduce the crutial notion of a \emph{decomposable solution}.
\begin{defn}\label{def_decomposale}\cite{etingof}
$(i)$ A subset $Y$ of a non-degenerate symmetric set-theoretical solution $X$ is said to be an \emph{invariant} subset if $S(Y \times Y)\subseteq Y \times Y$.\\
$(ii)$ An invariant subset $Y$  is said to be  \emph{non-degenerate} if $(Y,S\mid_{Y\times Y})$ is a non-degenerate and symmetric set.\\
$(iii)$ A non-degenerate and symmetric solution $(X,S)$ is said to be \emph{decomposable} if $X$ is a union of two nonempty disjoint non-degenerate invariant subsets. Otherwise,  $(X,S)$ is said to be \emph{indecomposable}.
\end{defn}
In \cite{etingof}, Etingof et al show that if $X$ is finite, then any invariant subset $Y$ of $X$ is non-degenerate. Moreover, they show that if  $(X,S)$ is non-degenerate and braided then the assignment $x \rightarrow f_{x}$ is a right action of $G(X,S)$ on $X$ and that $(X,S)$ is indecomposable if and only if $G(X,S)$ acts transitively on $X$. They give a classification of non-degenerate symmetric solutions with $X$ up to $8$ elements, considering their decomposability and other properties.
Gateva-Ivanova conjectured that every square-free, non-degenerate symmetric solution~$(X,S)$ is decomposable whenever $X$ is finite. This has been proved by Rump in~\cite{rump}. He also prove that the extension to an infinite set~$X$ is false.
Finally, in \cite{chou_art}, the first author finds a criterion for decomposability of the solution involving the Garside structure of the structure group. She proves that a non-degenerate symmetric solution $(X,S)$  is indecomposable if and only if its structure group is $\Delta$-pure Garside.
\section{Garside monoids and groups}
We turn now to the notion of a Garside group. We only recall the basic material that we need, and refer to~\cite{deh_francais}, \cite{deh_livre} and ~\cite{godelle} for more details.

\subsection{Garsideness}
We start with some preliminaries. If $M$ is a monoid generated by a set $X$, and if $x\in M$ is the image of the word~$w$ by the canonical morphism from the free monoid on $X$ onto $M$, then we say that \emph{$w$ represents $x$}. A monoid~$M$ is \emph{cancellative} if for every~${x,y,z,t}$ in $M$, the equality~$xyz = xtz$ implies ~$y = t$. The element $x$ is a \emph{left divisor}  ({\it resp.} a \emph{right divisor}) of $z$ if there is an element $t$ in $M$ such that $z=xt$ ({\it resp.} $z = tx$).  It is left noetherian ({\it resp.}  \emph{right noetherian}) if every sequence~$(x_n)_{n\in\mathbb{N}}$ of elements of $M$ such that $x_{n+1}$ is a left divisor ({\it resp.} a right divisor) of $x_n$ stabilizes. It is noetherian if it is both left and right noetherian. An element~$\Delta$ is said to be \emph{balanced} if it has the same set of right and left divisors. In this case, we denote by~$\Div{\Delta}$ its set of divisors. If~$M$ is a cancellative and noetherian monoid, then left and right divisibilities are partial orders on~$M$.
\begin{defn} (i) A \emph{locally Garside monoid} is a cancellative noetherian monoid such that any two elements have a common multiple for left-divisibility and right-divisibility if and only if they have a least common multiple for left-divisibility and right-divisibility, respectively.\\
(ii)  A \emph{Garside element} of a locally Garside monoid is a balanced element~$\Delta$ whose set of divisors~$\Div{\Delta}$ generates the whole monoid. When such an element exists, then we say that the monoid is a \emph{Garside monoid}.\\
(iii) A (\emph{locally}) \emph{Garside group}~$G(M)$ is the enveloping group of a (locally) Garside monoid~$M$.
\end{defn}
Garside groups have been first introduced in~\cite{DePa}. The seminal example are the so-called \emph{Artin-Tits groups}. Among these groups, \emph{Spherical type Artin-Tits groups} are special examples of Garside groups. We refer to~\cite{DiM} for general results on locally Garside groups. Recall that an element $x\neq 1$ in a monoid is called an \emph{atom} if the equality $x = yz$ implies $y = 1$ or $z = 1$. It follows from the definining properties of a Garside monoid that the following properties holds for a Garside monoid~$M$: The monoid~$M$ is generated by its set of atoms, and every atom divides the Garside elements. Any two elements in $M$ have a left and right gcd and lcm; in particular, $M$ verifies the Ore's conditions, so it embeds in its group of fractions~\cite{Clifford}. The left and right lcm of two Garside elements are Garside elements and coincide; therefore, by the noetherianity property there exists a unique minimal Garside element for both left and right
  divisibilities. This element will be called \emph{the} Garside element of the monoid in the sequel.

\begin{ex}\cite{chou_art}
\label{example_struct_gp} Consider the notation of Example~\ref{exemple:exesolu_et_gars}. The group $G(X,S)$ is a Garside group. The Garside element $\Delta$ is the right and left lcm of $X$, and is represented by $x_{1}^{4}x_{5}$,  $x_{2}^{4}x_{5}$, $x_{3}^{4}x_{5}$, $x_{4}^{4}x_{5}$, and others.
 \end{ex}
\subsection{Parabolic subgroups}
The notion of a parabolic subgroup of an Artin-Tits group is well-known. In~\cite{godelle}, the second author extends this notion to the wider context  of Garside groups.

\begin{defn}\label{def:parabsubgrp} Let $M$ be a Garside monoid with Garside element~$\Delta$.
\\(i) A submonoid $N$ of $M$ is \emph{standard parabolic} if there exists~$\delta$ in $\Div{\Delta}$ that is balanced and such that~$\Div{\delta}$ generates~$N$ with $N\cap \Div{\Delta} = \Div{\delta}$.
\\(ii) A \emph{standard parabolic} subgroup of the Garside group~$G(M)$ is a subgroup generated by a parabolic submonoid.
\end{defn}

In the sequel, we denote by $M_\delta$ the monoid~$N$ in the above definition.

\begin{lem} \cite{godelle} Let $M$ be a Garside monoid with Garside element~$\Delta$, and consider a standard parabolic submonoid~$M_\delta$. Then $M_\delta$ is a Garside monoid with~$\delta$ as a Garside element. Moreover, the Garside
 group~$G(M_\delta)$ is isomorphic to the parabolic subgroup of $G(M)$ generated by $M_\delta$.
\end{lem}

\begin{ex} If $M$ is an Artin-Tits monoid, then its classical parabolic submonoids are the parabolic submonoids defined by the associated Garside structure. More precisely, these are the submonoids generated by any set of atoms of $M$.
\end{ex}

\begin{ex}
Consider the notation of Example~\ref{exemple:exesolu_et_gars}. There are two non-trivial standard parabolic subgroups. One is generated by  $\{x_{1},x_{2},x_{3},x_{4}\}$, and  the other is generated by~$\{x_{5}\}$.
 \end{ex}

\section{Parabolic subgroups of the structure group}
The following result explains the deep connection between the theory of set-theoretical solutions of QYBE and that of  Garside groups.
\begin{thm}\cite[Thm.1]{chou_art}\label{thm_garsidetableau_structuregp2}
(i) Assume that  $\operatorname{Mon} \langle X\mid R \rangle$ is a  Garside monoid such that:
\begin{enumerate}
\item[(a)] the cardinality of $R$ is $n(n-1)/2$, where $n$ is the cardinality of $X$ and each side of a relation in $R$ has length 2 and
\item[(b)] if the  word $x_{i}x_{j}$ appears in $R$, then it appears only once.
\end{enumerate}
 Then, there exists a function $S: X \times X \rightarrow X \times X$ such that $(X,S)$ is  a non-degenerate symmetric set-theoretical solution and $\operatorname{Gp} \langle X\mid R \rangle$ is its structure group.\\
(ii) For every non-degenerate symmetric set-theoretical solution~$(X,S)$, the structure group $G(X,S)$ is a Garside group, whose Garside monoid is as above.
\end{thm}

Our objective in this section is to show that the connection is even deeper. Indeed, we prove Theorem~\ref{theo:theoprincip}. As a first step, we show that an invariant subset generates a standard parabolic subgroup (see Proposition~\ref{thm_invariant_parabolic}). In a second step, we prove that every such subgroup arises in this way (see Proposition~\ref{thm_parabolic_invariant}).
\subsection{From invariant subsets to standard parabolic subgroups}
For all this section, we assume $X$ is a finite set, and $(X,S)$ is a non-degenerate symmetric set-theoretical solution of the QYBE. For short, we write $G$ for $G(X,S)$. We denote by $M$ the submonoid of $G$ generated by $X$, and by $\Delta$  the Garside element of $M$. We recall that we denote by~$\Div{\Delta}$ the set of divisors of~$\Delta$. We fix an invariant subset~$Y$ of $X$, we denote by $\delta$ the right lcm of $Y$ in $M$, and by $M_Y$ and $G_Y$ the submonoid and the subgroup, respectively, of $G$ generated~$Y$.

\begin{prop}[\cite{chou_art}]\label{prop_simple_lcm}
(i) The Garside element~$\Delta$ is the lcm of $X$ for both left and right divisibilities.\\
(ii) Let $s$ belong to $M$. Then, \begin{center}\begin{tabular}{ll}&$s$ belongs to~$\Div{\Delta}$\\  $\iff$&$\exists X_{\ell}\subseteq X$ such that~$s$ is the right lcm of~$X_{\ell}$\\$\iff$& $\exists X_{r}\subseteq X$ such that~$s$ is the left lcm of~$X_r$.\end{tabular}\end{center}(iii) If~$s$ belongs to $\Div{\Delta}$ then the subsets $X_{\ell}$ and $X_r$ defined in Point~(ii) are unique and have the same cardinality.
\end{prop}
\begin{proof}
Point~(i) is proved in~\cite{chou_art}. Only the first equivalency of (ii) is proved in~\cite[Prop.~4.4]{chou_art}, but the second one is implicit there. Finally, Point~(iii) is also a direct consequence of~\cite{chou_art}:  Indeed, from the proof of \cite[Theorem 4.7]{chou_art}, the length of the right lcm of $k$ distinct elements of $X$ is equal to $k$  and in a dual way, the same result holds for the left lcm.
\end{proof}

 Note that~$X_{\ell}$ and~$X_{r}$ may be different. In the sequel, for~$s$ in $\Div{\Delta}$, we denote by~$X_{\ell}(s)$ and~$X_r(s)$ the subsets $X_{\ell}$ and $X_r$ defined in Lemma~\ref{prop_simple_lcm}. For instance, $X_{\ell}(\Delta) = X_r(\Delta) = X$, and $X_{\ell}(\delta) = Y$. It is interesting to note that the equality $X_{\ell}(s)=X_{r}(s)$ does not imply that $s$ is balanced. In Example \ref{example_struct_gp}, it holds that $x_{1}^{3}$ is the left and right lcm of the generators $x_{1}$, $x_{2}$, and $x_{3}$, but $x_{1}^{3}$ is not balanced since $x_{1}x_{2}$ is a left divisor but not a right divisor.
\begin{lem}\label{lem_invariant_balanced}
(i) If $s$ belongs to~$M_Y$, then all the  letters in  a word that represents~$s$ belong to~$Y$. In particular, every left or right divisor of $s$ lies in $M_Y$.\\
(ii) Let $s$ belong to $\Div{\Delta}$. Then $$s\in M_Y \iff X_{\ell}(s)\subseteq Y \iff X_r(s)\subseteq Y.$$
In particular, $\delta$ belongs to $M_Y$.\\
(iii) The monoid~$M_Y$ is equal to $M \cap G_{Y}$.
 \end{lem}
\begin{proof}
The following is the key argument in the proof: if $Y$ is invariant,  then $S(Y,Y)$ is included in $Y\times Y$, which means that the defining relations involving two generators from $Y$ in one-hand side of the relation have necessarily the form $y_{1}y_{2}=y_{3}y_{4}$, where $y_{i} \in Y$, $1 \leq i \leq 4$. Let $s$ belong to~$M_Y$, then Point~(i) follows directly from the above. Let us prove~(ii). If additionally $s$ belongs to $\Div{\Delta}$, then  $X_{\ell}(s)\subseteq Y$ and $X_{r}(s)\subseteq Y$. Conversely, if~$s \in \Div{\Delta}$, $X_{\ell}(s)\subseteq Y$ and $X_{r}(s)\subseteq Y$ then using the same argument and the \emph{reversing process}~(see \cite{deh_francais}) to compute the right  (or left) lcm of a subset of $Y$, we have $s\in M_Y$. It remains to show that~(iii) holds. Clearly, $M_Y$ is included in~$M \cap G_{Y}$ and by applying the double reversing process on a word on~$Y^{\pm 1}$ that represents an element in $M \cap G_{Y}$, we  obtain an element in $M$, whose letters belong to~$Y$  (still by the same argument).
\end{proof}

We recall that if $s$ is a balanced element in~$\Div{\Delta}$, then its \emph{support}~$\operatorname{Supp}(s)$ is defined to be the set~$X\cap \Div{s}$. It is shown in~\cite{godelle} that the atoms set of a standard parabolic subgroup $G_\delta$ is~$\operatorname{Supp}(\delta)$.
\begin{lem}\label{lem_invariant_balanced_2}
The element $\delta$ lies in~$\Div{\Delta}$, is balanced, and~$\operatorname{Supp}(\delta)=Y$.
\end{lem}
\begin{proof}
The right lcm of $Y$ is~$\delta$ and $Y\subseteq X$, so  by Proposition~\ref{prop_simple_lcm}(i), the element~$\delta$ is a left divisor of $\Delta$ and $\delta$ lies in~$\Div{\Delta}$. From lemma~\ref{lem_invariant_balanced}~(ii), the equality~$X_\ell(\delta) = Y$ implies that  $\delta$ lies in~$M_Y$ and~$X_r(\delta)\subseteq Y$. But the sets~$X_\ell(\delta)$ and~$X_r(\delta)$ have the same cardinality by Proposition~\ref{prop_simple_lcm}(iii), so we have~$X_\ell(\delta) = X_r(\delta) = Y$. Now, let~$u$ be a left divisor of~$\delta$. We show that~$u$ is also a  right divisor of~$\delta$. Since~$\delta$ belongs to~$\Div{\Delta}\cap M_Y$, it follows from lemma~\ref{lem_invariant_balanced}~(i) that~$u$ lies in~$\Div{\Delta}\cap M_Y$. So, by Lemma~\ref{prop_simple_lcm}(ii), we get~$X_r(u) \subseteq Y=X_r(\delta)$. Therefore, $u$ is a right divisor of~$\delta$.  Similarly every right divisor of $\delta$ is a left divisor of $\delta$, and $\delta$ is balanced. At last, we have~$\operatorname{Supp}(\delta)=Y$.
\end{proof}
\begin{lem} \label{lem_invariant_condition_parabolic}
 $\Div{\delta}=$ $\Div{\Delta} \cap M_{Y}$.
\end{lem}
\begin{proof} By Lemmas ~\ref{lem_invariant_balanced} and~\ref{lem_invariant_balanced_2}, $\delta$ lies in~$\Div{\Delta} \cap M_{Y}$, so~$\Div{\delta}$ is included in~$\Div{\Delta} \cap M_{Y}$ ({\it cf.} Lemma~\ref{lem_invariant_balanced}(ii)). Conversely, if~$u$ lies in~$\Div{\Delta} \cap M_{Y}$, then~$u$ is the right lcm of~$X_\ell(u)$, which is a subset of $Y$ by  Lemma~\ref{lem_invariant_balanced}(ii). Since~$\delta$ is the right lcm of~$Y$, the element $u$ belongs to~$\Div{\delta}$.
\end{proof}
We can now state and prove the main result of this section:
\begin{prop}\label{thm_invariant_parabolic}
  Under the general hypothesis and notations of this section, the subgroup~$G_{Y}$, generated by $Y$, is a  standard parabolic subgroup of $G$.
\end{prop}
\begin{proof}
By Lemma~\ref{lem_invariant_balanced_2}, $\delta$ is a balanced element in $\Div{\Delta}$, with $Y= \operatorname{Supp}(\delta)$. Moreover, from Lemma~\ref{lem_invariant_condition_parabolic}, $\operatorname{Div}(\delta)=$ $\operatorname{Div}(\Delta) \cap M_Y$. Hence, $G_{Y}$ is a  standard parabolic subgroup of $G$ from Definition~\ref{def:parabsubgrp}.
\end{proof}
Rump proves that every square-free, non-degenerate and symmetric solution $(X,S)$  is decomposable, whenever $X$ is finite \cite{rump}. So, from Proposition \ref{thm_invariant_parabolic}, the structure group of a square-free solution has standard non-trivial parabolic subgroups. Moreover,  in a square-free solution each set $\{x\}$ with  $x\in X$ is an invariant subset, since $S(x,x)=(x,x)$. Square-free solutions provide examples of solutions in which there exist invariant subsets~$Y$ such that~$X \setminus Y$ is not invariant.

If we consider that~$Y_{1}$, $Y_{2}$,.., $Y_{k}$ are invariant subsets of $(X,S)$ that correspond to indecomposable solutions (in other words that $Y_{1}$, $Y_{2}$,.., $Y_{k}$ are the minimal under inclusion among  the invariant sets in the decomposition of $(X,S)$ to indecomposable solutions), then each set $Y_{1}$, $Y_{2}$,.., $Y_{k}$ generates a standard parabolic subgroup. Furthermore, these parabolic subgroups are  $\Delta$-pure Garside, since  a solution is indecomposable if and only if its structure group is $\Delta$-pure Garside  \cite[5.3]{chou_art}. So, we get that $G$ is the crossed product of  $\Delta$-pure Garside parabolic subgroups generated by $Y_{1}$, $Y_{2}$,.., $Y_{k}$,  using \cite[Prop.4.5]{picantin}.
\subsection{From parabolic subgroups to invariant sets}
As in the previous section, for all this section we assume $X$ is a finite set, and $(X,S)$ is a non-degenerate symmetric set-theoretical solution of the QYBE. For short, we write $G$ for $G(X,S)$. We denote by $M$ the submonoid of $G$ generated by $X$, and by $\Delta$ the Garside element of $M$. We recall that we denote by~$\Div{\Delta}$ the set of divisors of~$\Delta$. We fix a balanced element~$\delta$ in $\Div{\Delta}$ such that the subgroup $G_{\delta}$, generated by~$\Div{\delta}$, is a non-trivial standard parabolic subgroup of $G$. We set $M_\delta = M\cap G_\delta$, which is equal to the submonoid of $M$ generated by~$\operatorname{Supp}(\delta)$. Our objective here is to prove that ~$\operatorname{Supp}(\delta)$ is an invariant subset of $X$.

\begin{lem}\label{lem_delta_lcm_support}
The balanced element~$\delta$ is the right lcm and the left lcm of $\operatorname{Supp}(\delta)$. In particular, $\delta$ is \emph{the} Garside element of the Garside monoid~$M_\delta$.
\end{lem}
\begin{proof}
It is immediate that $X_\ell(\delta) = X_r(\delta) = \operatorname{Supp}(\delta)$ since $\delta$ is balanced ({\it cf.} see also \cite{godelle}).

\end{proof}

\begin{prop}\label{thm_parabolic_invariant}
Under the general hypothesis and notations of this section, $\operatorname{Supp}(\delta)$ is an invariant subset of $X$.
\end{prop}
\begin{proof}
Since $X$ is finite, from \cite{etingof} it is enough to show that $Y= \operatorname{Supp}(\delta)$ is invariant, that is $S(Y,Y) \subseteq (Y,Y)$. Let $y, y'$ belong to~$Y$, and assume  that  $S(y,y')=(x,x')$, that is $x$ is a left divisor of $yy'$ and $x'$ is a right divisor of $yy'$. From \cite[Prop.2.5]{godelle}, $M_\delta$ is closed under left and right divisibility, so $x,x'$ lie in $M_\delta$. As $x$ and $x'$ are atoms, they belong to $Y$. As a consequence, we get~$S(Y,Y) \subseteq (Y,Y)$.
\end{proof}

Gathering Propositions~\ref{thm_invariant_parabolic} and~\ref{thm_parabolic_invariant}, we get Theorem~\ref{theo:theoprincip}.

 We now extend the result of Proposition~\ref{thm_parabolic_invariant} to parabolic subgroups in general, that is not necessarily \emph{standard} parabolic subgroups. If $gG_{\delta}g^{-1}$ is the conjugate of a standard parabolic subgroup $G_{\delta}$, then $gG_{\delta}g^{-1}$ is generated by the set $gYg^{-1}$, where $Y$ generates $G_{\delta}$ and $Y$ is an invariant subset of $(X,S)$ from Proposition~\ref{thm_parabolic_invariant}. We recall that two solutions  $(X,S)$ and  $(X',S')$  are said to be \emph{isomorphic} if there exists a bijection $\phi: X \rightarrow X'$ which maps $S$ to $S'$, that is $S'(\phi(x),\phi(y))=(\phi(S_{1}(x,y)),\phi(S_{2}(x,y)))$.
\begin{prop}
Consider the general hypothesis and notations of this section. Let $gG_{\delta}g^{-1}$ be a non-trivial parabolic subgroup of $G$, where $g$ belongs to $G$. Then, $g\operatorname{Supp}(\delta)g^{-1}$ is an invariant subset of a set-theoretical $(X',S')$ which is isomorphic to $(X,S)$.
 \end{prop}
\begin{proof}
We define $(X',S')$ in the following way:\\ $X'=gXg^{-1}$ and $S'(gx_{i}g^{-1},gx_{j}g^{-1})=(gg_{i}(j)g^{-1},gf_{j}(i)g^{-1})$. \\
 From the definition,  $(X,S)$ and $(X',S')$ are isomorphic and a direct computation shows  that $g\operatorname{Supp}(\delta)g^{-1}$ is an invariant subset of $(X',S')$. If  $(X,S)$ and $(X',S')$ are isomorphic, then  $G$ and the structure group $G'$ of $(X',S')$ are isomorphic groups
(see for example \cite{chou}). Moreover, $gXg^{-1}\subseteq G$, so   $G'$ is a subgroup of $G$ that is isomorphic to $G$.
\end{proof}

\section{Garside subgroups of the structure group}
In the previous section, we investigated the connection between the invariant subsets of a set-theoretical solution and the standard parabolic subgroups of its structure group, equipped with is canonical Garside structure.  In~\cite{godelle}, the second author introduced a more general family of subgroups of a Garside group, namely the Garside subgroups. In this section, we study, in the case of a Garside group that arises as the structure group of a set-theroretical solution of the QYBE, how these subgroups can also be associated with set-theoretical solutions. Morover, we extend the notion of a decomposable solution, and explain why this extention could be useful in order to study the classification problem.

\subsection{The set-theoretical solution associated with a Garside subgroup}
Parabolic subgroups of Garside groups are \emph{Garside subgroups} as introduced in~\cite{godelle}. So, invariant subsets provide Garside subgroups. A question that arises naturally is whether the converse is true, that is whether a Garside subgroup is necessarily a parabolic one. The answer is negative as shown in Example \ref{example_struct_gp_indecomposable_2}. Indeed, a Garside subgroup may  not be generated by a subset of atoms.
\begin{defn}\label{defn_Garside_subgp} Let $M$ be a Garside monoid with $\Delta$ as Garside element. Let~$N$ be a submonoid of~$M$.\\(i) \cite[Prop.~1.6]{godelle} We say that $N$ is a \emph{Garside submonoid} of $M$ if it is generated by a non-empty subset~$D$ of $\Div{\Delta}$ which is a sublattice of $\Div{\Delta}$ for left and right divisibility, which is closed by complement for left and right divisibility, and such that for every $x,y\in D$, the left gcd and the right gcd of $xy$ and $\Delta$ in $M$ belong to $D$. In this case, we say that the subgroup of $G(M)$ generated by $N$ is a \emph{Garside subgroup} of $G(M)$.\\
(ii) we say that the \emph{Garside submonoid}~$N$ of $M$ is \emph{atomic} when its atom set is a subset of the atom set of $M$.
 \end{defn}

\begin{ex} \label{example_struct_gp_indecomposable_2} Consider the group~$G(X,S)$ where $X = \{x_1,x_2,x_3,x_4\}$ and the defining relations are $$\begin{array}{ccccc}
  x^{2}_{1}=x^{2}_{2};&x_{1}x_{2}=x_{3}x_{4};&
x_{1}x_{3}=x_{4}x_{2};\\ x^{2}_{3}=x^{2}_{4};&
x_{2}x_{4}=x_{3}x_{1};&
x_{2}x_{1}=x_{4}x_{3}.
\end{array}$$ The group~$G(X,S)$ has no proper standard parabolic subgroup. However, the subgroups generated by the sets $\{x_1\}$, $\{x_1^2\}$, $\{x_1,x_2\}$ and $\{x^2_1,x_3^2\}$ are examples of Garside subgroups.
\end{ex}
Since Garside subgroups are not necessarily parabolic, they may not correspond to a set-theoretical solution of the QYBE. However, the following result proves that under a simple necessary condition, a Garside subgroup is naturally associated to a set-theoretical solution of the QYBE.
\begin{prop}\label{thm_garsidesubgp_solution}
 Let $X$ be a finite set, and $(X,S)$ be a non-degenerate, symmetric set-theoretical solution of the QYBE. Assume $H$ is a Garside subgroup of $G(X,S)$ such that its atoms set~$X_H$ is closed under right complement in the Garside monoid~$M(X,S)$. Then there exists a uniquely well-defined bijective map~$S_H: X_H\times X_H\to X_H\times X_H$ such that the pair~$(X_H,S_H)$ is a non-degenerate symmetric set-theoretical solution and $G(X_H,S_H)$ is isomorphic to $H$.
\end{prop}

\begin{proof}
We claim that we  (uniquely) define a bijective map~$S_H: X_H\times X_H\to X_H\times X_H$ by setting~$S_H(x,y) = (z,t)$ for every~$x,y,z,t$ in $X_H$ with $x,z$ distinct and such that $xy = zt = x\lor z$.  By assumption, for every $x,z$ distinct in $X_H$, there exists a unique pair~$y,t$  of distinct elements in $X_H$ such that $xy = zt = x\lor z$, that is the map $(x,z)\mapsto (y,t)$ is injective. From finiteness, it is also bijective and  as a consequence, the set~$X_H$ is closed by left complements. Consider the group~$\tilde{H}$ with the generating set $X_H$ and the defining relations: $xy = zt$ whenever $xy = zt = x\lor z$. The presentation of~$\tilde{H}$ satisfies the properties (a) and (b) stated in Theorem~\ref{thm_garsidetableau_structuregp2}. So, the pair~$(X_H,S_H)$ is a non-degenerate symmetric set-theoretical solution such that  $G(X_H,S_H)\simeq \tilde{H}$.
At last, we show that the surjective morphism~$\varphi_H$ from $\tilde{H}$ to $H$ that sends $x\in X_H$ to itself is an isomorphism.  Let  $w$ be a word in $X_H^{\pm 1}$. Then the double reversing process in $G(X,S)$ (see \cite{deh_francais}) transforms~$w$ into a word that represents the same element in $\tilde{H}$, since $X_H$ is closed by left and right complements. So, if $w$  represents $1$ in $H$, then the double reversing process transforms $w$ into the trivial word in $\tilde{H}$, that is ~$\varphi_H$ is injective.
\end{proof}

As an immediate consequence we get:
\begin{cor}\label{cor_garsidesubgp_solution}
 Let $X$ be a finite set, and $(X,S)$ be a non-degenerate, symmetric set-theoretical solution of the QYBE. Assume $H$ is an atomic Garside subgroup of $G(X,S)$ with atoms set $X_H \subseteq X$. Denote by $S_H$ the restriction of $S$ to $X_H\times X_H$. The pair~$(X_H,S_H)$ is a non-degenerate symmetric set-theoretical solution and $G(X_H,S_H)$ is isomorphic to $H$.
\end{cor}
\begin{proof} Using Proposition \ref{thm_garsidesubgp_solution}, we  show that the set $X_H$ is closed under right complement. If $x,z$ belong to~$X_H$ and $S(x,y) = (z,t)$, then $xy = zt = x\lor z$ and $y,t$ belong to~$X$. From the definition of a Garside subgroup (see Defn. \ref{defn_Garside_subgp}), $H$ is generated by a lattice closed under right and left complement. So, $y,t$ belong to this generating lattice and also to ~$X$, so they belong to $X_H$.
\end{proof}
\begin{ex} Consider Example~\ref{example_struct_gp_indecomposable_2}. The atomic Garside subgroup of~$G(X,S)$ generated by $\{x_1,x_2\}$ is isomorphic to the structure group defined by the presentation~$\langle x_1,x_2\mid x_1^2 = x_2^2\rangle$.

\end{ex}

\subsection{Foldable solution}
The notions of an invariant subset and of a decomposable solution have been introduced in~\cite{etingof} as tools to build and classify set-theoretical solutions of the QYBE. In this last section, we want to explain how the notion of a Garside subgroup can be used in the same way. Let us first introduce the notion of a \emph{foldable solution}. We say that a partition $X_1\cup\cdots \cup X_k$ of a set~$X$ is \emph{proper} if each $X_i$ is not empty and $1<k<|X|$.
\begin{defn} Let $X$ be a finite set, and $(X,S)$ be  a non-degenerate, symmetric set-theoretical solution of the QYBE.\\
(i) We say that $(X,S)$ is \emph{foldable} if $X$ has a proper partition~$X_1\cup\cdots\cup X_k$ such that
\begin{enumerate}
\item Every set $X_i$ generates an atomic Garside subgroup of $M(X,S)$ with Garside element~$\Delta_i$;
\item The set $X' = \{\Delta_1,\cdots, \Delta_k\}$ is closed by right and left complements in $M(X,S)$ and is the atom set of a Garside subgroup of $G(X,S)$.
\end{enumerate}
(ii) In this case, let $S': X'\times X'\to X'\times '$ the bijective map induced by $S$. We say that $(X',S')$ is a \emph{folding} of $(X,S)$, or equivalently that $G(X',S')$ is a folding of~$G(X,S)$.\\
(iii) We say that $(X,S)$ is \emph{strongly foldable} when furthermore each $X_i$ generates a standard parabolic subgroup of $G(X,S)$. In this case, we say that~$(X',S')$ is a \emph{ strong folding} of $(X,S)$.
\end{defn}

 One can note that the Garside element of~$M(X',S')$ has to be equal to the Garside element of~$M(X,S)$. This notion is very similar to the notion of  \emph{folding} of a Coxeter graph used  by Crisp in~\cite{crisp} to study morphisms between two spherical type Artin-Tits groups. Indeed, in this case all the foldings have to be strong, and it is shown in~\cite{godelle} that such a morphism is characterized by conditions that are very closed to properties~$1$ and~$2$ stated in the above definition.

Foldable solutions exist, since  decomposable solutions are foldable as it will be shown below.

\begin{thm} Let $X$ be a finite set and $(X,S)$ be  a non-degenerate, symmetric set-theoretical solution of the QYBE. If $(X,S)$ is decomposable then it is strongly foldable, and the folding is a trivial solution to the QYBE. Moreover, $(X,S)$ is decomposable if and only of it has a strong folding~$(X',S')$ which is  a trivial solution and such that $\# X' = 2$.\label{Prop:prpfoldable}
 \end{thm}

In particular, we get Theorem~\ref{theo:theoprincip2}. When proving Theorem~\ref{Prop:prpfoldable}, we shall need some results proved in~\cite{picantin} and~\cite{chou_art}, and recalled below.

\begin{prop}Let $(X,S)$  be a decomposable non-degenerate, symmetric set-theoretical solution of the QYBE and let $M$ and $G$ be the corresponding  Garside monoid and group with $X$ as atoms set.  For $x$ in $X$, we set $\Delta_x = \vee \{b^{-1}(x\vee b) ; b \in M\}$.\\
(i)\cite{picantin, godelle2} The relation $\sim$ on $X$ defined by $x\sim y$ if $\Delta_x = \Delta_y$  is an equivalence relation. Moreover, if $\Delta_x \neq \Delta_y$ then $\Delta_x \land \Delta_y = \Delta_x \tilde{\land} \Delta_y = 1$ in $M$.\\
(ii)\cite{picantin, godelle2} Let $Y_1,\cdots, Y_k$ be the equivalence classes of $\sim$, and set $\Delta_i = \Delta_x$ for $x\in Y_i$. Then the subgroup of $G$ generated by $\Delta_1,\cdots,\Delta_k$ is a free abelian group with base $\Delta_1,\cdots,\Delta_k$.\\
(iii) \cite{chou_art} Assume  $X= \cup_{i=1}^{i=k} Y_i$, where $Y_i$ are invariant subsets of $X$ and $k \geq 2$, since the solution is decomposable. Then $\Delta = \Delta_1\cdots \Delta_k$, where $\Delta$ is the Garside element in $M$ and ~$\Delta_i$ is the lcm of $Y_i$. \label{prop:proppicetchou}
\end{prop}

\begin{proof}[Proof of Theorem~\ref{Prop:prpfoldable}]
Assume $(X,S)$  is decomposable and  $X= \cup_{i=1}^{i=k} Y_i$, where $Y_i$ are invariant subsets of $X$ and $k \geq 2$. From Proposition \ref{thm_invariant_parabolic}, the subgroup generated by $Y_i$  is a parabolic subgroup with Garside element $\Delta_i$, where $\Delta_i$ is the lcm of $Y_i$. So, the first condition in the definition of a strongly foldable solution is satisfied.
 Consider the set~$D = \{\Delta_1^{\varepsilon_1}\cdots \Delta_k^{\varepsilon_k}\mid \varepsilon_i = \{0,1\}\}$. Gathering several results from \cite[Sec.~2]{picantin} and Proposition~\ref{prop:proppicetchou}, we obtain that $D$ is a sublattice of $\Div{\Delta}$ and it is closed by complement and lcm: $$(\Delta_1^{\varepsilon_1}\cdots \Delta_k^{\varepsilon_k}) \land (\Delta_1^{\varepsilon'_1}\cdots \Delta_k^{\varepsilon'_k}) = \Delta_1^{\min(\varepsilon_1,\varepsilon'_1)}\cdots \Delta_k^{\min(\varepsilon_k,\varepsilon'_k)}$$ and $$(\Delta_1^{\varepsilon_1}\cdots\Delta_k^{\varepsilon_k}) \lor( \Delta_1^{\varepsilon'_1}\cdots \Delta_k^{\varepsilon'_k}) = \Delta_1^{\max(\varepsilon_1,\varepsilon'_1)}\cdots \Delta_k^{\max(\varepsilon_k,\varepsilon'_k)}.$$ Moreover, we have $$(\Delta_1^{\varepsilon_1+\varepsilon'_1 }\cdots\Delta_k^{\varepsilon_k+\varepsilon'_k}) \land \Delta = \Delta_1^{\max(\varepsilon_1,\varepsilon'_1)}\cdots \Delta_k^{\max(\varepsilon_k,\varepsilon'_k)} $$
Let $H$ be the subgroup  generated by $D$, then $H$ is a Garside subgroup of $G(X,S)$, with $\Delta$ as Garside element and atoms set $X'=\{\Delta_i;$ $1 \leq i \leq k\}$. That is, the second condition in the definition of a strongly foldable solution is also satisfied, so $(X,S)$ is strongly foldable. Furthermore,  $H$ is the structure group of a solution $(X',S')$ of the QYBE from Proposition \ref{thm_garsidesubgp_solution}, and  it is free abelian (from Proposition~\ref{prop:proppicetchou}(ii)), so  $(X',S')$ is the trivial solution on a set with $k$ elements. Thisa implies that  $(X,S)$ is strongly foldable and that the folding $(X',S')$  of $(X,S)$ is a trivial solution to the QYBE. It remains to show that it is possible to choose the folding $(X',S')$  such that $\# X' = 2$.
Let  $i$ in~$\{1,\cdots,k\}$ such that $Y_i$ and $Z_i = \cup_{j\neq i}Y_i$ are two invariant subsets that generate parabolic subgroups whose Garside elements are $\Delta_i$ and $\Delta_{\hat{i}} = \prod_{j\neq i}\Delta_i$, respectively.  We have $\Delta_i\Delta_{\hat{i}} = \Delta_{\hat{i}}\Delta_i = \Delta$.  The subgroup~$H'_i$ of $G(X,S)$ generated by $\Delta_i$ and $\Delta_{\hat{i}}$ satisfies $H'_i\cap \Div{\Delta} = \{1,\Delta_i,\Delta_{\hat{i}},\Delta \}$. Hence, $H'_i$ is a strong folding of $G(X,S)$ that corresponds to a trivial solution. Conversely, let $(X',S')$ be a strong folding which is a trivial solution and such that $\# X' = 2$. Then,  then by definition, $X'$ contains two elements that are the right (and left) lcm of two invariant subsets of $X$,  $Y_1$ and $Y_2$, that satisfy  $X = Y_1 \cup Y_2$. Therefore, $X$ is decomposable.
\end{proof}

Decomposable solutions of the QYBE are foldable. However, there exist foldable solutions that are not decomposable.
\begin{ex} \label{exemplefoldingpasstrong}
Consider the set-theoretical solution~$(X',S')$ whose structure group is defined by the presentation~$\langle x,y\mid x^2 = y^2\rangle$. Then, $(X',S')$ is a folding of the solution $(X,S)$ defined in Example~\ref{example_struct_gp_indecomposable_2}.
Indeed, consider the following partition of $X$:  $X = \{x_1,x_2\}\cup\{x_3,x_4\}$. The sets $\{x_1,x_2\}$ and $\{x_3,x_4\}$ generate  Garside subgroups with Garside elements $x_1^2$ and $x_3^2$ respectively. The set
$X'=\{x = x_1^2,$ $y = x_3^2\}$ is closed under right and left complement and generates a Garside subgroup of $G(X,S)$. So, $(X',S')$ is a folding of $(X,S)$ but this folding is not strong, since the Garside subgroups generated by  $\{x_1,x_2\}$ and $\{x_3,x_4\}$ are not parabolic through they are atomic.
\end{ex}

As seen in Theorem~\ref{Prop:prpfoldable} and Example~\ref{exemplefoldingpasstrong}, foldable and strongly foldable solutions do exist. Conversely the notion of folding can be used as a tool to build new solutions: starding from a solution~$(X',S')$ one can try to \emph{substitute} each element~$x$ of~$X'$ by the Garside element of another solution~$(X_x,S_x)$ so that one obtains a new solution~$(\cup_{x\in X'} X_x,S)$ which is foldable, with $(X',S')$ as a folding. Clearly, the difficult point is to define a bijective map~$S$ that is compatible with~$S'$ and so that one gets a solution to the QYBE.

As a final comment, we raise the question of the existence of a strongly foldable solution~$(X,S)$ that is not decomposable.




\bigskip\bigskip\noindent
{ Fabienne Chouraqui,}

\smallskip\noindent
Bar-Ilan University, Ramat Gan, Israel.
\smallskip\noindent
E-mail: {\tt fabienne@tx.technion.ac.il}

\bigskip\bigskip\noindent
{ Eddy  Godelle,}

\smallskip\noindent
Universit\'e de Caen, UMR 6139 du CNRS, LMNO, Campus II, 14032 Caen cedex, France.
\smallskip\noindent
E-mail: {\tt eddy.godelle@unicaen.fr}


\end{document}